\documentclass[12pt, leqno]{amsart}
\usepackage[OT2,T1]{fontenc}
\DeclareSymbolFont{cyrletters}{OT2}{wncyr}{m}{n}
\DeclareMathSymbol{\Sha}{\mathalpha}{cyrletters}{"58}
\usepackage{indentfirst}
\usepackage{amstext}
\usepackage{amsopn}
\usepackage{amsfonts}
\usepackage{amsmath}
\usepackage{latexsym}
\usepackage{amscd}
\usepackage{amssymb}
\usepackage{amsmath}
\usepackage[all,cmtip]{xy}
\usepackage{leftidx}
\usepackage{graphicx}
\usepackage{tikz}
\usepackage{ulem}
\usepackage{hyperref}

=\oddsidemargin \font\teneufm=eufm10 \font\seveneufm=eufm7
\font\fiveeufm=eufm5
\newfam\eufmfam
\textfont\eufmfam=\teneufm \scriptfont\eufmfam=\seveneufm
\scriptscriptfont\eufmfam=\fiveeufm

\def\frak#1{{\fam\eufmfam\relax#1}}
\let\goth\mathfrak

\def\cB{\mathcal B}

\def\m{\mathfrak m}

\def\cP{\mathcal P}
\def\gP{\mathfrak{p}}
\def\cU{\mathcal U}

\def\cp{\frak p}

\def\GG{\mathbb{G}}

\def\gG{\goth G}

\def\gV{\goth V}

\def\gX{\goth X}

\def\gp{\goth p}

\def\1{\mbox{\bf 1}}

\DeclareMathOperator{\Hom}{Hom}
\DeclareMathOperator{\Aut}{Aut}

\DeclareMathOperator{\uAut}{\underline{Aut}}
\DeclareMathOperator{\uIsom}{\underline{Isom}}

\DeclareMathOperator{\Out}{Out} 
\DeclareMathOperator{\Isom}{Isom}

\DeclareMathOperator{\Ind}{Ind}
\DeclareMathOperator{\per}{per}

\DeclareMathOperator{\Spin}{\rm Spin}

\DeclareMathOperator{\SO}{\rm SO}
\DeclareMathOperator{\PGL}{\rm PGL}
\DeclareMathOperator{\GL}{\rm GL}
\DeclareMathOperator{\SL}{\rm SL}

\newcommand{\incl}[1][r]
{\ar@<-0.2pc>@{^(-}[#1] \ar@<+0.2pc>@{-}[#1]}

\newtheorem{stheorem}{Theorem}[section]

\newtheorem{scorollary}[stheorem]{Corollary}
\newtheorem{slemma}[stheorem]{Lemma}
\newtheorem{sproposition}[stheorem]{Proposition}
\newtheorem{sremark}[stheorem]{Remark}
\newtheorem{sremarks}[stheorem]{Remarks}

\newtheorem{ssetting}[stheorem]{Setting}

\theoremstyle{definition}

\numberwithin{equation}{section}

\def\ZZ{\mathbb{Z}}

\def\gE{\mathfrak{E}}

\def\gG{\mathfrak{G}}

\def\gP{\mathfrak{P}}
\def\gQ{\mathfrak{Q}}

\def\Par{\mathrm{Par}}

\def\ol{\overline}

\def\fppf{\text{\rm fppf}}

\def\2int{\mathop{2\int}\nolimits}

\def\Spec{\mathop{\rm Spec}\nolimits}

\def\Hom{\mathop{\rm Hom}\nolimits}

\def\Ind{\mathop{\rm Ind}\nolimits}

\def\Gal{\mathop{\rm Gal}\nolimits}

\def\Pic{\mathop{\rm Pic}\nolimits}
\def\Ind{\mathop{\rm Ind}\nolimits}

\def\Br{\mathop{\rm Br}\nolimits}
\def\Aut{\text{\rm{Aut}}}
\def\Out{\text{\rm{Out}}}
\def\sm{\smallskip}

\def\Par{\text{\rm{Par}}}

\def\Isom{\mathop{\rm Isom}\nolimits}

\def\Nrd{\mathop{\rm Nrd}\nolimits}

\def\resp.{\mathop{\rm resp.}\nolimits}

\def\limind{\mathop{\oalign{lim\cr
\hidewidth$\longrightarrow$\hidewidth\cr}}}
\def\Ker{\mathop{\rm Ker}\nolimits}

\def\Res{\mathop{\rm Res}\nolimits}
\def\Cor{\mathop{\rm Cor}\nolimits}

\def\lgr{\longrightarrow}

\font\math=cmmi10
\def\varpi{\hbox{\math\char'44}}

\def\simlgr{\buildrel\sim\over\lgr}

\def\pa{\S\kern.15em }

\def\un{\uppercase\expandafter{\romannumeral 1}}
\def\deux{\uppercase\expandafter{\romannumeral 2}}
\def\trois{\uppercase\expandafter{\romannumeral 3}}
\def\quatre{\uppercase\expandafter{\romannumeral 4}}
\def\cinq{\uppercase\expandafter{\romannumeral 5}}
\def\six{\uppercase\expandafter{\romannumeral 6}}

\def\hfl#1#2#3{\smash{\mathop{\hbox to#3{\rightarrowfill}}\limits
^{\scriptstyle#1}_{\scriptstyle#2}}}
\def\gfl#1#2#3{\smash{\mathop{\hbox to#3{\leftarrowfill}}\limits
^{\scriptstyle#1}_{\scriptstyle#2}}}

\title[Local-global principle]{Local-global principle for over semiglobal fields}

\author[P. Gille]{Philippe Gille}
\thanks{The first author was supported by the project "Group schemes, root systems, and related representations" founded by the European Union - NextGenerationEU through Romania's National Recovery and Resilience Plan (PNRR) call no. PNRR-III-C9-2023-
I8, Project CF159/31.07.2023, and coordinated by the Ministry of Research, Innovation and Digitalization (MCID)
of Romania. }
\address{UMR 5208 Institut Camille Jordan - Universit\'e Claude Bernard Lyon 1
43 boulevard du 11 novembre 1918
69622 Villeurbanne cedex - France}  
\address{and Institute of Mathematics "Simion Stoilow" of the Romanian Academy,
21 Calea Grivitei Street, 010702 Bucharest, Romania}
\email{gille@math.univ-lyon1.fr}

\author[R. Parimala]{Raman Parimala}\address{Departement  of Mathematics and Computer Science,
MSC W401, 400 Dowman Dr. Emory University
Atlanta, GA 30322
USA  
}
\email{parimala.raman@emory.edu}

\date{\today}

\begin{document}

 \begin{abstract}  We compare different  local-global principles for torsors 
 under a reductive group $G$ defined over a semiglobal field $F$.
 In particular if the $F$--group $G$ is a retract rational $F$--variety,
 we prove that the local global principle holds for the completions
 with respect to divisorial valuations of $F$.
  
\smallskip


\end{abstract}

\maketitle


\bigskip

\section{Introduction}\label{section_intro}

The theme is the arithmetic theory of torsors
over a semiglobal field. It was initiated
by Hartmann, Harbater and Krashen \cite{HHK}
and developed further \cite{CTOHHKPS2, CTPS, HHK2,PS, RS}.

Let $B$ be a  complete discrete valuation  ring with fraction field $K$
and residue field $k$. Let $p \geq 1$
be the characteristic exponent of $k$. Let $X$ be a smooth, projective, geometrically 
integral curve over $K$. Let $F=K(X)$ be the function field of $X$ and let $t$ be an uniformizing parameter of $B$.
Let $\gX$ be a normal model of $F$, i.e.\ a normal connected
projective $B$-curve with function field $F$.
 We denote by $Y$ the closed fiber of $\gX$ and fix a separable closure $F_s$ of $F$.

Given an affine algebraic $F$--group $G$, we can 
study the  local-global principle for $G$-torsors
with respect 
to the infinite set of overfields $F_P$ that arise from completions at the points $P$ of $Y$. This gives rise to the Tate-Shafarevich set
$$
\Sha_{\gX}(F,G)=
\ker \Bigl( H^1(F,G) \to \prod_{P \in Y} 
H^1(F_P,G)\Bigr);
$$
it is denoted simply by $\Sha_{X}(F,G)$ in \cite{HHK2}.
Using Lipman's resolution of singularities permits to
consider the inductive limit
$$
\Sha_{patch}(F,G) = \limind_{\gX} \Sha_\gX(F,G)
$$
taken on the models (or equivalently on the regular models). 
On the other hand, we 
consider the completions associated to the set $\Omega_{F, div}$
of divisorial valuations on $F$ arising of various
regular projective models over $B$ of the curve  $X$.
It gives rise to the 
$$
\Sha_{div}(F,G)=
\ker \Bigl( H^1(F,G) \to \prod_{v \in \Omega_{F, div}} 
H^1(F_v,G)\Bigr);
$$
In view of \cite[prop.\ 8.2]{HHK2}, we have the following inclusion 
$$
\Sha_{patch}(F,G) \subseteq \Sha_{div}(F,G).
$$
\sm
The main result of the paper is to show (under some condition on $p$ with respect to  $G$),
that the  two  Tate-Shafarevich sets coincide (Thm.\ \ref{thm_main}).
If furthermore $G$ is an $F$-rational variety (or even
retract rational $F$--variety), 
we have  that $\Sha_{patch}(F,G)=1$ 
according to
Harbater-Krashen-Hartmann \cite[Thm.\ 9.1]{HHK2} (see
\cite[Thm.\ 2.2.4]{K} for the extension to the retract 
rational case)
so  that $\Sha_{div}(F,G)=1$ in this case.
It means that if $G$ is a retract rational $F$--variety
then the local-global principle principle holds for $G$-torsors with respect to divisorial valuations on $F$. 
A further related  question is whether  
 the set $\Sha_{div}(F,G)$ coincide with the set  
 $\Sha_{patch}(F,G)$  
with respect to a single integral model and a single patch on the model.  
 A first result in that direction is the case of inner type $A_{n-1}$ ($n$ prime to $p$)
 obtained recently by Suresh \cite[Thm.\ 1.4]{Su}.
 
We discuss also a $0$-cycle version of this result, see Theorem \ref{thm_main_cycle}.
This was known for certain classical groups
\cite{RS,PS, Su}. 
As in \cite{GP}, the proof uses  the study
of  loop torsors which requires
group schemes defined over $B$. To do so, for a given Chevalley group
$G_0$ over $\ZZ$, 
we explain in the first section why the $G_0\rtimes \Aut(G_0)$-torsors
permit to classify couples $(E,G)$ where $G$ is a form
of $G_0$ and $E$ a $G$-torsor.

\sm

\medskip

\noindent{\bf Acknowledgements.} We thank Jean-Louis Colliot-Th\'el\`ene for valuable comments
on a preliminary version.  

 \bigskip  
 
 \smallskip


\section{Pairs of a group scheme together with  a torsor}

The following is a complement to the ``sorites" of \cite{G3}. We work in the setting of fppf sheaves on a base scheme $S$.
We consider the groupoid $\gP$ whose objects are pairs $(E,G)$
where $G$ is a fppf $S$--sheaf in groups and $E$ a $G$--torsor; the morphisms between two pairs  $(E_1,G_1)$ and  $(E_2,G_2)$
are  pairs of isomorphisms $f=(f', f'')$ where $f'': G_1 \simlgr G_2$ is an $S$--isomorphism of $S$-sheaves in groups and ${f': E_1 \simlgr E_2}$ an $S$-isomorphism
of sheaves  such that the following diagram commutes

\[\xymatrix{
 E_1 \times_S G_1 \ar[d]^{f' \times f''} \ar[r] & E_1 \ar[d]^{f'}    \\ 
 E_2 \times_S G_2 \ar[r] & E_2 .
}\]
Note that it induces an isomorphism  $f_*: \,  \uAut_{G_1}(E_1) \simlgr \, \uAut_{G_2}(E_2)$ or
in other words  an isomorphism $f_*: \, {^{E_1}G_1}\simlgr \, {^{E_2}G_2}$
between the associated twisted group sheaves.

\begin{slemma}\label{lem_pair} Let $G$ be an $S$--sheaf in groups and let 
$E$ be a $G$--torsor over $S$. We denote by $^EG$
the twisted $S$--sheaf of $G$ by $E$ by inner automorphisms.

\smallskip

\noindent (1) We have a natural exact sequence of 
 fppf $S$--sheaves in groups
$$
1 \to {^EG} \to \uAut_\gP\bigl(E,G\bigr)
\to \uAut(G) \to 1
$$ 
which is locally split for the flat topology.

\smallskip

\noindent (2) If $E=G$ the sequence above splits and the 
 $S$--functor $\uAut_\gP\bigl( G,G\bigr)$ is isomorphic  to $G \rtimes \uAut(G)$.
\end{slemma}

\begin{proof}
We prove the two assertions at the same time.
The   projection $\uAut_\gP\bigl(E,G\bigr)
\to \uAut(G)$ is a homomorphism of $S$--sheaves in groups.
On the other hand we have a monomorphism 
${^EG} = \uAut_G(E) \to \uAut_\gP\bigl( E,G\bigr)$, 
$f'' \mapsto (f'', id)$. The fact that the sequence is exact at the middle term is 
the fact \cite[cor. 2.3.8]{Gd}.

We consider now the case $E=G$. In this case the sequence above splits
by $f'' \mapsto (f'',f'')$ hence an isomorphism 
$G \rtimes \uAut(G) \simlgr \uAut_\gP\bigl(G,G \bigr) $.
 Coming back to the general case it 
implies that the projection $\uAut_\gP\bigl( E,G \bigr)
\to \uAut(G)$ is an epimorphism of $S$--sheaves.
\end{proof}

Given  a fppf $S$--sheaf in groups $G$, we can apply \cite[lemme 2.6.3]{G3}
to express the cohomology

\begin{equation} \label{eq_coh}
H^1(S, G \rtimes \uAut(G)) \simlgr \bigsqcup_{ [G']  \in H^1(S,\uAut(G))}
H^1(S,G')/\uAut(G')(S).
\end{equation}

\noindent by using the fact that $H^1(S, \uAut(G))$ classifies 
the $S$--forms of $G$.
The action on $\uAut(G')(S)$ on $H^1(S,G')$ is the natural one
so that the map $H^1(S,G') \to H^1(S,G')/\uAut(G')(S)$ has trivial kernel 
so that the composite
$H^1(S,G') \to H^1(S, G \rtimes \uAut(G))$
has trivial kernel. This map is nothing but the composition
$$
H^1(S,G') \to H^1(S,G' \rtimes \uAut(G'))
 \xrightarrow[\sim]{ \hbox{Torsion bijection}}
H^1(S, G \rtimes \uAut(G))
$$
 where the torsion bijection is with respect to the $\uAut(G)$--torsor
 $\uIsom(G,G')$.
Summarizing the decomposition \eqref{eq_coh}, 
$H^1(S, G \rtimes \uAut(G))$ encodes the classes of torsors
for all twisted forms of $G$.
The map $\mathrm{int}: G \to \uAut(G)$ extends
to 

\[\xymatrix{
u:  G \rtimes \uAut(G)  & \ar[r] &   \uAut(G) \rtimes^{c} \uAut(G)  
 & \ar[r]^{\sim \qquad \quad} & \uAut(G) \times \uAut(G) ,\\ 
 (g,f) & \mapsto & (\mathrm{int}(g), f) \\
 &&  (f_1, f_2)& \mapsto&  (f_1 f_2, f_2)
}\]
 where $c$ stands for the conjugacy action.
 In terms of objects classified by the respective torsors, 
 $u_*$  applies a pair $(E',G')$ to  $(^{E'}G', G')$.

\section{Refinements on the reduction of torsors to finite subgroups}

We provide two  variants of \cite[Theorem 1.2.(2)]{CGR}  for reduction of  torsors
from certain group schemes to finite subgroup schemes.
Let $1 \to G \to \widetilde G \to J \to 1$
be an exact of $\ZZ$--group schemes where 
$G$ is a Chevalley reductive $\ZZ$--group scheme and
$J$ a constant $\ZZ$--group scheme (not necessarily finite).
Let $T \subset G$ be a maximal $\ZZ$-split subtorus
and put $N=N_G(T)$, $W=N_G(T)/T$.

\begin{slemma} \label{lem_rep} (1) The fppf sheaf normalizer
$\widetilde N=N_{\widetilde G}(T)$ is representable
by a smooth $\ZZ$--group scheme which fits in an exact 
sequence
$$
1 \to N \to \widetilde N \to J \to 1.
$$

\noindent (2) The fppf quotient $\widetilde W=\widetilde N/T$ is representable by a constant $\ZZ$-group scheme. Furthermore the map $\widetilde N(\ZZ) \to 
\widetilde W(\ZZ)$ is onto.

\end{slemma}

\begin{proof}
(1) According to \cite[lemme 3.4.5.(1)]{G3},
$\widetilde N$ is representable by a $\ZZ$--group 
scheme which is locally of finite presentation.
We claim that the map $\widetilde N \to J$ is an epimorphism
of fppf sheaves. We are given a ring $A$ and an element
$x \in J(A)$ and want to find a fppf cover $B$ of $A$
such that $x_B \in J(B)$ lifts to $\widetilde G(B)$.
The preimage of $x$ in $\widetilde G_A$
is a $G_A$--torsor so is an affine smooth $A$--scheme 
$\Spec(B_1)$. It follows that there exists $\widetilde g_1 \in \widetilde G(B_1)$ mapping to $x_{B_1} \in J(B_1)$.
Next we consider the maximal $B_1$--torus
$^{\widetilde g_1}\! T$ of $G_{B_1}$.
We consider the strict transporter 
$\mathrm{Transpst}_{G_{B_1}}( T_{B_1}, \, ^{\widetilde g_1}\! T)$
as defined in \cite[VI$_B$.6]{SGA3}. This the
$B_1$-functor defined by
$$
\mathrm{Transpst}_{G_{B_1}}( T_{B_1}, \, ^{\widetilde g_1}\! T)(C)
= \bigl\{ g \in G(C) \, \mid \,  g T_{B_1}(C')g^{-1} = \, (^{\widetilde g_1}\! T)(C')
\enskip \hbox{for each $C$--ring $C'$}  \bigr\}
$$
for each $B_1$--ring $C$.
According to Grothendieck \cite[XI.5.2]{SGA3},  it  is representable
by an $N$--torsor over $\Spec(B_1)$ so is an affine  smooth $B_1$--scheme  $\Spec(B_2)$.
It follows that there exists $g_2 \in G(B_2)$
such that    $^{g_2}\!T_{B_1} = \, ^{\widetilde g_1}\! T$
so that $\widetilde n_2=g_2^{-1} g_{1, B_2} \in \widetilde N(B_2)$.
Thus $\widetilde n_2$ maps to $x_{B_2} \in J(B_2)$
and the claim is established.
We have then an exact sequence of $\ZZ$--group schemes
$1 \to N \to \widetilde N \to J \to 1$.
Since $N$ and $J$ are smooth, 
so is $\widetilde N$ in view of \cite[VI$_B$.9.(xii)]{SGA3}.

\smallskip

\noindent (2) A diagram chase provides an 
exact sequence of fppf $\ZZ$-sheaves
$1 \to W \to \widetilde W \to J \to 1$.
Since $\ZZ$ is simply connected, we have 
$H^1(\ZZ, W)=1$ whence an exact sequence
\begin{equation} \label{eq_weyl}
 1 \to W(\ZZ) \to \widetilde W(\ZZ) \to J(\ZZ) \to 1.
 \end{equation}
We denote by $\widetilde W_0$ the constant $\ZZ$--group
scheme attached to the abstract group  $\widetilde W(\ZZ)$.
It fits in the  exact sequence of constant $\ZZ$-group schemes 
$1 \to W \to \widetilde W_0 \to J \to 1$. 
We have a homomorphism $\widetilde \phi: \widetilde W_0 \to \widetilde W$ which is an isomorphism by diagram chasing.
Thus $\widetilde W$ is representable by a 
constant $\ZZ$-group scheme.
\end{proof}

Our goal is to construct a family $(\widetilde S_{2n})_{n \geq 1}$ of 
$\ZZ$--subgroups schemes of $\widetilde N$.
We have an action of $\widetilde W$ on $T \cong \GG_{m,\ZZ}^r$
whence a homomorphism of constant $\ZZ$--group schemes
$h: \widetilde W  \to \Aut_\ZZ(T)  \cong \GL_r(\ZZ)_\ZZ$.
We define the constant $\ZZ$--group schemes
$T_2= \bigl( T(\ZZ)\bigr)_\ZZ$,
$\widetilde N_2= \bigl( \widetilde N(\ZZ)\bigr)_\ZZ$. 
Since $T(\ZZ)= \{\pm 1 \}^r$, we observe $T_2 \cong (\ZZ/2\ZZ)^r$
so that the map $T_2 \to T$ is induced by the $\ZZ$--map 
$\ZZ/2\ZZ \to \mu_2$, $1 \mapsto -1$.
The sequence \eqref{eq_weyl} gives rises 
to  exact sequence of constant $\ZZ$--group schemes
$1 \to T_2 \to \widetilde N_2 \to \widetilde W \to 1$
which fits in a commutative exact diagram

\[\xymatrix{
 1 \ar[r]  & T_2    \ar[r] \ar[d] & \widetilde N_2
 \ar[r] \ar[d] & \widetilde W  \ar[r]   \ar@{=}[d] &1  \\ 
 1 \ar[r]  & T    \ar[r]  & \widetilde N
 \ar[r]  & \widetilde W  \ar[r]    &1 .
}\]
For each $n \geq 1$, we define the $\ZZ$--group scheme
$\widetilde S_{2n}$  as the push-out of  $\widetilde N_2$ by the  homomorphism $T_2 \to {_2T} \to {_{2n}T}$, it fits in 
in a commutative exact diagram
\[\xymatrix{
 1 \ar[r]  & T_2    \ar[r] \ar[d] & \widetilde N_2
 \ar[r] \ar[d] & \widetilde W  \ar[r]   \ar@{=}[d] &1  \\ 
 1 \ar[r]  & {_{2n}T}    \ar[r] \ar[d] & \widetilde S_{2n} \ar[d]
 \ar[r]  & \widetilde W  \ar[r]  \ar@{=}[d]  &1 \\
 1 \ar[r]  & T    \ar[r]  & \widetilde N
 \ar[r]  & \widetilde W  \ar[r]    &1 
}\]  
where the map ${_{2n}T} \to T$ is the natural embedding.
We observe that $\widetilde S_{2n}$ is a $\ZZ$--subgroup scheme of $\widetilde N$.
We denote by $d(\widetilde G)$ the l.c.m.\ of
the cardinals of  finite subgroups of $\mathrm{Im}(h) \subset \GL_r(\ZZ)$.

\newpage

\begin{stheorem}\label{thm_cgr1}
Let $X$ be a scheme. We assume one of the following:

\sm

(i) $X$ is a connected normal scheme which is locally noetherian;

\sm

(ii) The image of $h: \widetilde W(\ZZ) \to \GL_r(\ZZ)$ is finite.

\smallskip

\noindent We assume furthermore  that there exists an integer $D \geq 1$ such that $$D \, \Pic(X')=0$$
for any finite \'etale cover $X'$ of $X$ whose degree divides $d(\widetilde G)$.
We consider the $\ZZ$--subgroup 
 $\widetilde S= \widetilde S_{2 d(\widetilde G) D}$
 of $\widetilde N$.

\sm

\noindent (1) The map $H^1_{\fppf}(X, \widetilde S) \to H^1_{\fppf}(X, \widetilde N)$ is onto.

 \sm

\noindent (2) Assume furthermore that  $X= \Spec(R)$ with $R$ an LG ring.
Then the map $H^1_{\fppf}(X, \widetilde S_{2 d(\widetilde G)}) \to H^1_{\fppf}(X, \widetilde G)$ is onto.

\end{stheorem}

Recall that a ring $R$ has the property LG if whenever a polynomial 
$f\in  R[x_1, \dots, x_n]$ represents a unit over $R_\m$ for each maximal ideal $\m$ of $R$,  then $f$ represents a unit over $R$ \cite{EG}. 
We need the following auxiliary statement.

\begin{slemma} \label{lem_torus} In the setting
of Theorem \ref{thm_cgr1}, let $F$ be a $\widetilde W$--torsor over $X$ and
consider the twisted $X$--torus $M= {^FT}$ with respect to the action of 
$\widetilde W$ on $T$ via $h$.

\sm
\noindent (1) $M$ is split after a finite \'etale extension
$X'/X$ of degree $d$ dividing $d(\widetilde G)$.

\sm

\noindent (2)
We have $d(\widetilde G) D \, H^1_{\fppf}(X,M)=0$.

\sm

\noindent (3)
For each $n\geq 1$, the map
$$
\Ker\Bigl( H^2_{\fppf}(X, {_nM}) \to H^2_{\fppf}(X, M)
\Bigr) \to H^2_{\fppf}(X, {_{n d(\widetilde G) D}M})
$$
is trivial.
\end{slemma}

\begin{proof}(1) We need to distinguish the cases (i) and (ii)
of Theorem \ref{thm_cgr1}.

\sm

\noindent{\it Case (i).}
According to \cite[lemme 2.14.(2)]{G6},
the $\widetilde W$--torsor $F$ is isotrivial.
Let $Y \to X$ be a finite connected  Galois  \'etale  cover
of group $\Gamma$ which splits $F$.
Then the class of $[F]$ is given by a homomorphism
$\phi: \Gamma \to \widetilde W(\ZZ)= \widetilde W(Y)$
and the isomorphism class of $M$ is given by the composite
$\Gamma \to \widetilde W(\ZZ) \xrightarrow{h} \GL_r(\ZZ)$.
Moding out by the kernel of this homomorphism provides
a finite connected  Galois \'etale cover $X'$ of $X$
of degree $d$ dividing $d(\widetilde G)$ and which splits $M$.

\sm

\noindent{\it Case (ii).} The scheme $X$ is not assumed connected anymore
but $\mathrm{Im}(h)$ is assumed to be finite.
It follows that  the $X$--torus $M$ is isotrivial,
it splits after  a finite connected  Galois  \'etale  cover
$X'/X$ of degree $d$ which divides $d(\widetilde G)$.

\sm

\noindent (2) By restriction-corestriction, the kernel
$\ker\bigl( H^1_{\fppf}(X,M) \to H^1_{\fppf}(X',M)  \bigr)$
is killed by $d$. On the other hand, we have
$M_{X'} \cong T_{X'} \simlgr \GG_{m,X'}^r$
so that  $H^1_{\fppf}(X',M)  \cong \Pic(X')^r$.
It follows that $D H^1_{\fppf}(X',M)=0$ whence
$d(\widetilde G)\,  D \, H^1_{\fppf}(X,M)=0$.

\smallskip

\noindent (3)  We consider the commutative diagram of exact 
sequences
\[\xymatrix{
 1 \ar[r]  & {_nM}    \ar[r] \ar[d] & M
 \ar[r]^{\times n} \ar[d]_{\times d(\widetilde G)\,  D} & M\ar[r]   \ar@{=}[d] &1  \\ 
 1 \ar[r]  & {_{n d(\widetilde G)\,  D}M}    \ar[r] & 
 \qquad M \qquad
 \ar[r]^{\qquad \times n d(\widetilde G)\,  D \qquad}  & M \ar[r]  &1.
}\]
It induces a commutative diagram of exact sequences
\[\xymatrix{
 H^1_{\fppf}(X,M) \ar[d]_{\times d(\widetilde G)\,  D}  \ar[r]  & H^2_{\fppf}(X,{_nM})    \ar[r] \ar[d] & H^2_{\fppf}(X,M)  \ar@{=}[d]   \\ 
 H^1_{\fppf}(X,M)  \ar[r]  & H^2_{\fppf}(X,{_{n d(\widetilde G)\,  D}M}  )   \ar[r] & H^2_{\fppf}(X,M)  .
}\]
Since the left handside vertical map is trivial by (1),
we obtain (3) by diagram chasing.
\end{proof}

We proceed to the proof of Theorem \ref{thm_cgr1}.

\begin{proof}
(1) Let $E$ be a $\widetilde N$-torsor 
and consider the $\widetilde W$--torsor
 $F= E \wedge^{\widetilde N} \widetilde W$. We consider
 the following commutative diagram of exact sequences

\[\xymatrix{
 1 \ar[r]  & {_2T}    \ar[r] \incl[d] & \widetilde S_2 
 \ar[r] \incl[d] & \widetilde W \ar[r] \ar@{=}[d] &1  \\ 
 1 \ar[r]  & {_{2 d(\widetilde G) \, D}T}    \ar[r] \incl[d] & \widetilde S_{2 d(\widetilde G) \, D} 
 \ar[r] \incl[d] & \widetilde W \ar[r] \ar@{=}[d] &1  \\ 
 1 \ar[r]  & T   \ar[r]  & \widetilde N 
 \ar[r]  & \widetilde W \ar[r] &1 . \\ 
}\]
The obstruction to lift $F$ to a $\widetilde S_2$--torsor
is a class $\Delta([F]) \in H^2_{\fppf}(X, {_2M})$
where the $X$--torus $M={^FT}$ is the twist of $T_X$ by $F$ \cite[IV.4.2]{Gd}.
The point is that its image in $H^2_{\fppf}(X,M)$
vanishes since $F$ lifts to the $\widetilde N$--torsor
$E$ in view of the compatibilities of the obstructions
\cite[\S IV.3.6]{Gd}. According to Lemma \ref{lem_torus}.(3),
we have
$\Delta([F]) \in \ker\Bigl( H^2_{\fppf}(X, {_2M})
\to H^2_{\fppf}(X,   {_{2 d(\widetilde G) \, D}M})\Bigr)$.
Again by functoriality, it follows that
the $\widetilde W$--torsor $F$ lifts to
a $\widetilde S_{2 d(\widetilde G) \, D}$-torsor
$V$. We consider the $\widetilde N$-torsor $E_V=  \widetilde N \rtimes^{\widetilde S_{2 d(\widetilde G) \, D}} V$.
Twisting by $V$ and $E_V$, the first diagram provides
the commutative diagram of exact 
sequence of $X$--group schemes

\begin{equation}  \label{diag18}
\xymatrix{
 1 \ar[r]  & {_{2 d(\widetilde G) \, D}M}    \ar[r] \incl[d] & {^V\!\widetilde S}_{2 d(\widetilde G) \, D} 
 \ar[r] \incl[d] & {^F \widetilde W} \ar[r] \ar@{=}[d] &1  \\ 
 1 \ar[r]  & M   \ar[r]  & {^{E_V}\!\widetilde N} 
 \ar[r]  & {^F\!\widetilde W} \ar[r] &1  .
}
\end{equation}

\noindent Next we consider the commutative diagram of torsion bijections
(with respect to the torsors $V, E_V$ and $F$)
\[\xymatrix{
H^1_{\fppf}(X, {^V\!\widetilde S}_{2 d(\widetilde G) \, D} ) \ar[r]^{\tau_2}_\sim \ar[d] &
H^1_{\fppf}(X, \widetilde S_{2 d(\widetilde G) \, D} ) \ar[d] \\
H^1_{\fppf}(X, {^{E_V}\!\widetilde N}) \ar[r]^{\tau}_\sim  \ar[d]  & H^1_{\fppf}(X, { \widetilde N})  \ar[d] \\
H^1_{\fppf}(X, {^F\!\widetilde W}) \ar[r]^{\tau'}_\sim  & H^1_{\fppf}(X, \widetilde W) .
}\]
Since $[E]$ maps to $[F]$ by
the map $H^1_{\fppf}(X, \widetilde N) \to H^1_{\fppf}(X, { \widetilde W})$, it follows that
 the class $\tau^{-1}([E])$ maps to $1 \in  H^1_{\fppf}(X, {^F\!\widetilde W})$.
 On the other hand, the diagram \eqref{diag18} induces a commutative diagram of exact sequences of pointed sets

\begin{equation}  \label{diag19}
\xymatrix{
 H^1_{\fppf}(X , {_{2 d(\widetilde G)\, D}M})     \ar[r] \ar[d] &
 H^1_{\fppf}(X,{^V\!\widetilde S}_{2 d(\widetilde G) \, D})
 \ar[r] \ar[d] &
 H^1_{\fppf}(X,{^F\!\widetilde W}) \ar@{=}[d]\\
  H^1_{\fppf}(X ,M  )  \ar[r]  & H^1_{\fppf}(X , {^{E_V}\!\widetilde N} )\ar[r]  &
 H^1_{\fppf}(X,{^F\!\widetilde W})
}
\end{equation}
and the point is that the left vertical map is onto in view of Lemma \ref{lem_torus}.(2).
A diagram chase shows that $\tau^{-1}([E])$ arises from a class of
$H^1_{\fppf}(X,{^V\!\widetilde S}_{2 d(\widetilde G) \, D})$. By applying $\tau$,
we obtain that  $[E]$ arises from a class of
$H^1_{\fppf}(X,{\widetilde S}_{2 d(\widetilde G) \, D})$.

\smallskip

\noindent (2) We assume now that $X=\Spec(R)$ with $R$ a LG ring.
According to \cite[Thm.\ 2.6]{GN}, all twisted forms of $\widetilde G$
admit a maximal $R$--torus so that
the map  $H^1(R,N_{\widetilde G}(T)) \to
H^1(R, \widetilde G)$ in onto.
The ring $R$ satisfies that $\Pic(R')=0$ for each
finite extension $R'/R$ \cite[Thm.\ 2.10]{EG}.
The condition on Picard groups is satisfied with $D=1$
so that (1) implies the wished statement.
\end{proof}

Let $l$ be a prime. Our goal is to provide
a localized version at $l$ of Theorem \ref{thm_cgr1}
when $J$ is a finite constant $\ZZ$-group so that
$\widetilde W$ is finite as well in view of Lemma \ref{lem_rep}.(1).
For $n \geq 1$, we need to construct a localized version at $l$
of the $\ZZ$--group scheme $\widetilde S_{2n}$ fitting in the sequence
\begin{equation}\label{eq_2n}
1 \to {_{2n}T} \to \widetilde S_{2n} \to \widetilde W \to 1. 
\end{equation}
We denote by $\widetilde W(\ZZ)^{(l)}$ a $l$-Sylow subgroup of
$\widetilde W(\ZZ)$ and by $\widetilde W^{(l)} \subset \widetilde W(\ZZ)$
the associated constant $\ZZ$-group scheme.  The first step is to 
pull back the above sequence by $\widetilde W^{(l)}  \to \widetilde W$ so that we obtain 
\begin{equation} \label{eq_bemol}
1 \to {_{2n}T} \to \widetilde S_{2n} \to \widetilde W^{(l)} \to 1. 
\end{equation}

\noindent{\it Case $l=2$.}
We write $n=2^a m$ with $m$ odd and 
define $\widetilde S_{2n}^{(2)}$ as the  push-out of the extension
\eqref{eq_bemol} with respect to  the map ${_{2n}T} \to {_{2^{a+1}}T}$. 
It is representable by an affine  $\ZZ$--group scheme\footnote{We mod out 
$\widetilde S_{2n} \times_\ZZ {_{2^{a+1}}T}$ by a finite diagonalizable $\ZZ$--subgroup.}
which is a $\ZZ$--subgroup scheme of $\widetilde S_{2n}$ and a fortiori of
$\widetilde N$. 
According to \cite[VI$_B$.9.2.(viii)]{SGA3}, the affine $\ZZ$--group scheme 
$\widetilde S_{2n}^{(2)}$ is flat and finite free.  

\begin{equation} \label{eq_diese}
1 \to {_{2^{a+1}}T} \to \widetilde S_{2n}^{(2)} \to \widetilde W^{(2)} \to 1. 
\end{equation}

\noindent{\it Case $l$ is odd.}
We write $n=l^b m$ with $(m,l)=1$  and 
define $\widetilde S_{2n}^{(l)}$ as the  push-out of the extension
\eqref{eq_bemol} with respect to  the map ${_{2n}T} \to {_{l^b}T}$.
As above it is representable by a finite flat affine $\ZZ$--group scheme
(which is a $\ZZ$--subgroup scheme of $\widetilde S_{2n}$ and a fortiori of
$\widetilde N$)
 and fits in an commutative diagram of exact sequences

\begin{equation} \label{eq_becart}
\xymatrix{
 1 \ar[r]  & {_{l^b}T}    \ar[r] \ar[d] & \widetilde S_{2n}^{(l)}
 \ar[r] \ar[d] & \widetilde W^{(l)}\ar[r]   \ar[d] &1  \\ 
 1 \ar[r]  & {_{2n}T}    \ar[r] & \widetilde S_{2n} \ar[r] &  \widetilde W 
 \ar[r]    &1 .
}
\end{equation}

We come back to the general case and 
 consider the quotient $\ZZ$-scheme $\widetilde S_{2n}/ \widetilde S^{(l)}_{2n}$
as defined in \cite[\S III.2.6]{DG}. It is an affine 
$\ZZ$-scheme and the quotient map 
$\widetilde S_{2n} \to \widetilde S_{2n}/ \widetilde S_{2n}^{(l)}$ is
 an $\widetilde S_{2n}^{(l)}$-torsor.
According to \cite[VI$_B$.9.2.(xi), (xii)]{SGA3},
$\widetilde S_{2n}/ \widetilde S_{2n}^{(l)}$ is flat and quasi-finite over $\ZZ$.
According to \cite[Tag 0AH6]{Stacks}, $\widetilde S_{2n}/ \widetilde S_{2n}^{(l)}$ is proper 
over $\ZZ$ so is finite \cite[Tag 02OG]{Stacks}.
Since a finite flat $\ZZ$--module is free, 
$\widetilde S_{2n}/ \widetilde S_{2n}^{(l)}$ is finite free over $\ZZ$
of degree $\deg\bigl(\widetilde S_{2n}) / \deg(\widetilde S_{2n}^{(l)} \bigr) $
 which is prime to $l$.

\begin{stheorem}\label{thm_cgr2} In the above setting, we assume
that $\widetilde W$ is finite over $\ZZ$ (equivalently $J$ is finite).
We assume furthermore  that 
there exists an integer 
$D$  such that $$ D \, \Pic(X')\{l\}=0$$
for any finite \'etale cover $X'$ of $X$ whose degree divides $d(\widetilde G)$.
We consider the $\ZZ$--subgroup 
 $\widetilde S^{(l)}= \widetilde S_{2 d(\widetilde G) D}^{(l)}$
 of $\widetilde N$.

\sm

\sm

\noindent (1) Let $E$ be a $\widetilde N$--torsor.
Then there exists a finite locally free flat cover $X'$ of $R$ of constant degree
prime to $l$ such that 
the $G$--torsor $E_{X'}$ admits a reduction to  
 $\widetilde S^{(l)}$.

\sm

\noindent (2) Assume that $X=\Spec(R)$ where $R$ is a LG ring.
Let  $E$ be a  $\widetilde G$--torsor. Then there exists
a finite locally free flat cover $R'$ of $R$ of constant degree
prime to $l$ such that 
the $G$--torsor $E_{R'}$ admits a reduction to  
 $\widetilde S_{2 d(\widetilde G)}^{(l)}$.
 
 \sm

\noindent (3)
If $R$ is a field, we can require furthermore that $R'$ is  a field in (2).
 
\end{stheorem}

We proceed  to the proof of Theorem \ref{thm_cgr2}.

\begin{proof} (1) According to Theorem \ref{thm_cgr1}.(1), 
the $\widetilde N$-torsor $\widetilde E$ admits
a reduction to a $\widetilde S_{2 d(\widetilde G)D}$-torsor
$F$. We consider the $\ZZ$-scheme
$$
X'=F/ \widetilde S_{2 d(\widetilde G)D}^{(l)}= 
F \wedge^{\widetilde S_{2 d(\widetilde G)D}} \widetilde S_{2 d(\widetilde G)D}/ \widetilde S_{2 d(\widetilde G)D}^{(l)}
$$
of reductions of $F$ to $\widetilde S_{2 d(\widetilde G)D}^{(l)}$.
By descent, it is finite locally free of degree
$\deg\bigl( \widetilde S_{2 d(\widetilde G)D}/\widetilde S_{2 d(\widetilde G)D}^{(l)} \bigr)$ which is prime to $l$.
It follows that $E_{X'}$ admits a reduction to $\widetilde S_{2 d(\widetilde G)D}^{(l)}$.

\sm

\noindent (2) The argument is similar
by applying this times Theorem \ref{thm_cgr1}.(2).

\sm

\noindent (3) We assume that $R$ is a field $k$.
Then the basic Lemma \ref{lem_coprime} below provides 
a field extension $k'/k$ of degree prime to $l$
such that $X'(k') \not = \emptyset$. This is the reason
why we can improve (2) in this case.
\end{proof}

\begin{slemma}\label{lem_coprime} Let $l$ be a prime and let 
$k$ be a field and let $X$ be a finite  $k$-scheme of degree $d$
prime to $l$. Then there exists a field extension $k'/k$
satisfying the three following facts:

\sm

(i) $([k':k], l)=1$;

\sm

(ii) $[k':k] \leq d$;

\sm

(iii) $X(k') \not=\emptyset$.
 
\end{slemma}

\begin{proof} According to \cite[Tag 0AAX]{Stacks}, we have $k[X]=A_1 \times \cdots \times A_c$ where $A_i$ is a (finite) Artinian $k$--algebra of residue field  $k_i$. We have $X(k_i) \not = \emptyset$
for each $i$ and  $d= \sum\limits_{i=1}^c [k_i:k]^{e_i}$.
Since $l$ and $d$ are coprime, there exists $i$ such that 
$[k_i:k]$ is prime to $l$.
\end{proof}

\section{Patching techniques}
 
\subsection{Setting}  \label{subsec_setting} 
We recall the setting.
Let $B$ be a  complete (excellent) discrete valuation ring with fraction field $K$, residue field $k$ and
uniformizing parameter $t$. 
Let $F$ be a one-variable function field over $K$
and let $\gX$ be a normal model of $F$, i.e.\ a normal connected
projective $B$-curve with function field $F$.
 We denote by $Y$ the closed fiber of $\gX$ and fix a separable closure $F_s$ of $F$.
   
 For each point $P \in Y$, let $R_P$ be the local ring of
 $\gX$ at $P$; its completion $\widehat R_P$ is a domain 
 with fraction field denoted by $F_P$.

 For each subset $U$ of $Y$ that is contained
 in an irreducible component of $Y$ and does not meet the other components, 
we define
 $R_U= \bigcap\limits_{P \in U} R_P \subset F$. We denote by $\widehat R_U$
 the $t$--adic completion of $R_U$.
 The rings $R_U$ and $\widehat R_U$ are excellent normal domains 
 and we denote by $F_U$ the fraction field of $\widehat R_U$
 \cite[Remark 3.2.(b)]{HHK3}.
 
 Each height one prime $\cp$ in $\widehat R_P$ that contains
  $t$ defines a branch of $Y$ at $P$ lying 
 on some irreducible component of $Y$. 
 The $t$-adic completion  $\widehat R_{\cp}$ of the local ring
 $R_{\cp}$ of $\widehat R_P$ at $\cp$ is a complete DVR
 whose fraction field is denoted by $F_{\cp}$. 
 The field $F_{\cp}$ contains also $F_U$ if $U$ 
 is an irreducible open subset of $Y$ such that $P \in \ol{U} \setminus U$.
 We have then a diagram of fields
 
\[\xymatrix{
 &   F_{\cp}  & \\ 
 F_P \ar[ru] & & F_U  . \ar[lu]
}\]

 \begin{ssetting} \label{setting_hhk} {\rm
 Let $\cP$ be a non-empty finite set of closed points of $Y$ that contains
 all the closed points at which distinct irreducible components meet
 and such that each component contains at least one point of $\cP$.
 Let $\cU$ be the set of connected components of $Y \setminus \cP$ and let $\cB$
 be the set of branches of $Y$ at points of $\cP$.
 This yields a finite inverse system of field $F_P, F_U, F_\gp$ (for $P \in \cP$;
 $U \in \cU$, $\gp \in \cB$) where $F_P, F_U \subset F_\gp$ if $\gp$
 is a branch of $Y$ at $P$ lying in the closure of $U$.
 }
\end{ssetting}

\subsection{Torsors and models}

\begin{sproposition} \label{prop_key}
Let $1 \to G \to \widetilde G \to J \to 1$
be an exact sequence of smooth $B$--group schemes
where $G$ is reductive and $J$ is twisted constant\footnote{
i.e.\ there exists an abstract group $\Gamma$ such that 
$J$ is locally isomorphic to the $B$--constant group scheme  $\Gamma_B$.}.
Let $E_1$ (resp.\ $E_2$) be a $\widetilde G$-torsor over $F$
admitting a reduction to a finite \'etale $B$--subgroup of 
$S_1$ (resp.\ $S_2$) of $\widetilde G$ with  prime to $p$ order. 

Assume that $E_{1,F_v} \cong E_{2,F_v}$ for any
divisorial discrete valuation $v$ on $F$.
Then there exists a regular model 
$\gQ$ of $F$ as above together with $Y, \cP$, etc.\ 
such  that 
$E_{1,F_P} \cong E_{2,F_P}$ for all $P \in \cP$
and  $E_{1,F_U} \cong E_{2,F_U}$ for all  
connected components $U$ of $Y \setminus \cP$.
\end{sproposition}
 
\begin{proof}
We consider the $F$-scheme 
$Z= \Isom_{\widetilde G}(E_1,E_2)$.
Our assumption is that $Z(F_v) \not = \emptyset$
for all divisorial discrete valuations $v$ on $F$.

Let $V_i$ be an $S_{i,F}$--torsor 
such that $V_i \wedge^{S_{i,F}} \widetilde G \cong E_i$
for $i=1,2$.
 According to \cite[lemma 5.1]{GP} 
 there exists  a regular proper model $\gX$ of $X$ and 
  a strict normal crossing divisor  $D$ containing  the irreducible components of $Y$ such that $V_i$ extends 
  to an $\gX \setminus D$--torsor $\gV_i$  under $S_i$.
  We are given a closed point $P \in Y$. If $y$ is of codimension one, by hypothesis, $Z(F_y)$ is not empty. We therefore look at a closed point $P \in Y$.
 We  pick a height one prime $\cp$ in $\widehat R_P$ that contains  $t$. It  defines a branch of $Y$ at $P$ lying 
 on some irreducible component $Y_1$ of $Y$.

We consider the regular local ring $A=\widehat R_P$ and denote by $A_D$
its localization at $D$. 
Since $S_{i,B}$ is finite \'etale of degree prime 
to $p$, $H^1(A_D,S_i)$ consists in loop torsors as defined in \cite[\S 2.3, lemma 2.3.(2)]{G5},  i.e.\ those arising from cocycles related to tame Galois covers of $A_D$. It follows that the $A_D$--torsors $\gV_1$, $\gV_2$ under $\widetilde G$ are loop ({\it ibid}, lemma 5.3.(3)).

Let $F_{P,v}$ be the completion of the field $F_P$ for the valuation associated to the blow-up of $\Spec(A)$ at its closed point.
Our assumption states in particular that $Z(F_{P,v}) \not = \emptyset$.
According to \cite[Thm.\ 6.9]{G6},  the map
$$
H^1_{loop}(A_D, \widetilde G) \to H^1(F_{P,v}, \widetilde G)
$$
is injective. It follows that 
$\Isom_{\widetilde G_{A_D}}\bigl( \gV_{1},\gV_{2}     \bigr)(A_D) \not = \emptyset$.
A fortiori $Z(F_P) \not = \emptyset$.
 \end{proof}

 \begin{sproposition} \label{prop_main} Let $G$ be a reductive $F$--algebraic group
 and let $G_0$ be the underlying Chevalley 
 reductive $\ZZ$--group scheme. Let $E$ be a $G$--torsor
 and assume that the $\Aut(G_0)$-torsor
 $\uIsom(G_0,G)$ admits  a reduction to  a finite \'etale 
$B$--subgroup of $ \Aut(G_0)$ of prime to $p$ order
and that the 
  $G_0 \rtimes \Aut(G_0)$-torsors $(G,G)$ (resp.\
 $(E,G)$) admits a reduction to a finite \'etale 
$B$--subgroup of $G_0 \rtimes \Aut(G_0)$ of prime to $p$ order. 
Then the two following assertions are equivalent:

 \smallskip
 
 \noindent (i) There  exists a regular proper model $\gX$ of $X$ 
with special fiber $Y=\gX_k$ such that for every point
$y \in Y$, then $E(F_y) \not = \emptyset$.

 \smallskip
 
 \noindent (ii) For each divisorial discrete valuation $v$
 on $F$, we have  $E(F_v) \not = \emptyset$.
 
 \end{sproposition}

 \begin{proof}
The implication $(i) \Longrightarrow (ii)$ is a general fact,
see \cite[Thm.\ 3.4]{HHKP}. We prove then
 $(ii) \Longrightarrow (i)$  assuming that $E(F_v) \not = \emptyset$ for 
 each discrete divisorial valuation $v$ on $F$.
 
 We denote by $s_0: \Aut(G_0) \to \Aut_{\gP}(G_0,G_0)$
  the splitting $\Aut_{\gP}(G_0,G_0) \to \Aut(G_0)$
  and consider the $\Aut_{\gP}(G_0,G_0)$-torsor $E_1= \uIsom_{\gP}\bigl( (G_0,G_0),(G,G) \bigr)$ which is nothing but $s_{0,*}\bigl( \uIsom_{gr}(G_0,G)\bigr)$.
  Our assumption implies that admits $E_1$ 
admits a reduction to  a finite \'etale 
$B$--subgroup of $\Aut_{\gP}(G_0,G_0)$ of prime to $p$ order.

We shall apply Proposition  \ref{prop_key} to the $G_0  \rtimes \Aut(G_0)$--torsors 
$E_1= \uIsom_{\gP}\bigl( (G_0,G_0),(G,G) \bigr)$ and $E_2=
\uIsom_{\gP}\bigl( (G_0,G_0),(E,G) \bigr)$.

Assume that $E_{1,F_v} \cong E_{2,F_v}$ for any
divisorial discrete valuation $v$ on $F$.
Proposition  \ref{prop_key} provides 
 a regular model 
$\gQ$ of $F$ as above together with $Y, \cP$, etc.  
such  that 
$E_{1,F_P} \cong E_{2,F_P}$ for all $P \in \cP$
and  $E_{1,F_U} \cong E_{2,F_U}$ for all  
connected components $U$ of $Y \setminus \cP$.
 It follows that for every point
$y \in Y$, then $E(F_y) \not = \emptyset$.
 \end{proof}

 \begin{stheorem} \label{thm_main} 
 Let $G$ be a reductive
  $F$--algebraic group
 and let $G_0$ be the underlying Chevalley 
reductive $\ZZ$--group scheme. 
Assume that $p$ does not divide the order of 
the Weyl group $W_0$ of $G_0$ and that $G$ becomes an inner form of $G_0$
after a  Galois prime to p extension of $F$.

\smallskip

(1) We have $\Sha_{patch}(F,G) =\Sha_{div}(F,G)$;

\smallskip

(2) If $G$ is a retract rational $F$--variety, then 
$\Sha_{div}(F,G)=1$.
\end{stheorem}

\begin{proof}
(1) We know that  $\Sha_{patch}(F,G) \subset  \Sha_{div}(F,G)$.
To establish the reverse inclusion, we
are given a $G$--torsor $E$ such that $E(F_v) \not = \emptyset$ for each divisorial discrete valuation $v$.
We need to check the conditions for applying Proposition \ref{prop_main}. 
We start with the isomorphism class $[G] \in H^1(F,\Aut(G_0))$
by using the sequence $1 \to G_{0,ad} \to \Aut(G_0) \to \Out(G_0) \to 1$
where $\Out(G_0)$ is a constant $\ZZ$--group scheme.
Since $H^1(F,\Out(G_0))= \Hom_{ct}\bigl( \Gal(F_s/F), \Out(G_0)(\ZZ) \bigr)$
modulo  $\Out(G_0)(\ZZ)$-conjugacy, the image of $[G]$ is represented by 
a continuous homomorphism $u : \Gal(F_s/F) \to  \Out(G_0)(\ZZ)$
whose kernel defines a Galois extension $F'/F$ 
which is the minimal subextension of $F_s$
making $G$ an inner form. Our assumption implies
that $\Gal(F'/F)$ is of prime to $p$ degree
and so  that $\Gamma=u(\Gal(F_s/F) ) $ is a finite subgroup
of $\Out(G_0)(\ZZ)$ of order prime to $p$. We denote by  $\Aut_{\Gamma}(G_0)=  \Aut(G_0) \times_{\Out(G_0)} \Gamma$, we have an exact 
sequence of $\ZZ$--group schemes
\begin{equation}\label{eq_Gamma}
1 \to G_{0,ad} \to \Aut_{\Gamma}(G_0) \to \Gamma_\ZZ \to 1.
\end{equation}
Further the isomorphism class  $[G]$
arises from a cohomology class $\gamma \in H^1(F,\Aut_\Gamma(G_0))$.
According to \cite[Thm.\ 1.2]{CGR} applied to the ring $\ZZ$,
there exists a finite $\ZZ$--subgroup $S$ of 
$\Aut_\Gamma(G_0)$ such that the map 
$H^1(F,S) \to H^1(F,\Aut_\Gamma(G_0))$ is onto.
Let $T_0\subset G_0$ be a maximal split $\ZZ$--subtorus
and consider the  $\ZZ$--group  scheme
$\Aut_\Gamma(G_0,T_0)$ of $\Aut_\Gamma(G_0)$ as defined 
in \cite[\S XXIV.3]{SGA3}. 
For each ring $C$, we have 
$$
\Aut_\Gamma(G_0,T_0)(C) = \bigl\{ f \in \Aut_\Gamma(G_0)(C) \, \mid \,
f(T_{0,C})=T_{0,C}  \bigr\}.
$$
Further by the same  statement, we can take $S \subset 
\Aut_\Gamma(G_0,T_0)$  fitting in an exact sequence $1 \to {_n(T_0)} \to S \to 
\Aut_\Gamma(G_0,T_0)/T_0 \to 1$ where $n$ divides the 
square of the order of the  finite group $\Aut_\Gamma(G_0,T_0)/T_0$.
We have a sequence
$1 \to W_0 \to \Aut_\Gamma(G_0,T_0)/T_0 \to \Gamma \to 1$
so that $\Aut_\Gamma(G_0,T_0)/T_0$ is finite constant of
 prime to $p$ order. It implies that  ${_n(T_0)}$ is 
 a finite diagonalisable $\ZZ$--group scheme
 whose base change to $B$ is \'etale.  
  Altogether  
 $S_B$ is a finite \'etale $B$-subgroup of $\Aut_\Gamma(G_0)_B$
 of prime to $p$ order such that $\gamma$ admits a reduction 
 to $S_B$. A fortiori, $[G] \in H^1(F,\Aut(G_0))$
 admits a reduction to $S_B$.
 
 The same method works as well for the 
 class $[(E,G)] \in H^1( F, G_0 \rtimes \Aut(G_0))$
 by using the exact sequence
 $1 \to G_0 \rtimes_\ZZ G_{0,ad}  \to G_0 \rtimes_\ZZ \Aut(G_0)
 \to \Out(G_0) \to 1$ and provides a finite \'etale 
$B$-subgroup of $S'$ of  $(G_0 \rtimes_\ZZ \Aut(G_0))_B$
of prime  to $p$-order such that the $G_0 \rtimes_\ZZ \Aut(G_0)$-torsor
over $\Spec(F)$  arises from an $S'$--torsor.

 Proposition \ref{prop_main}, $(ii) \Longrightarrow (i)$, applies and shows that  there  exists a regular proper model $\gX$ of $X$ 
with special fiber $Y=\gX_k$ such that for every point
$y \in Y$, then $E(F_y) \not = \emptyset$.
In other words $[E]$ belongs in $\Sha_{\gX}(F,G)$
and a fortiori to $\Sha_{patch}(F,G)$.

\smallskip

\noindent (2) If $G$ is a retract rational
$F$--variety, we have that $\Sha_{patch}(F,G)=1$  \cite[Thm.\ 9.1]{HHK2}
(see \cite[Thm.\ 2.2.4]{K} for the extension to the retract 
rational case).
The first part of the statement enables us to 
conclude that $\Sha_{div}(F,G)=1$ in this case.
\end{proof}

 \begin{sremarks}{\rm 
(a) If $G$ is a torus, the condition is that 
the minimal splitting extension of $T$ is of 
prime to $p$ degree.  Note that for norm one  tori
with respect to a prime to $p$ Galois extension,  the main theorem
had been established by Sumit Chandra Mishra \cite[Thm.\ 5.1]{M}.

\smallskip
 
 \noindent (b) The assumptions on  $p$ are not sharp.
 For example, if $G$ is semisimple of type $G_2$,
 the statement excludes the primes $2$ and $3$ but excluding $2$
 is enough. The point is that we can deal with the subgroup
 $\mu_2 \times \mu_2 \times \ZZ/2\ZZ \subset S$ in the proof.

 Similarly for $n \geq 1$ such that $(p, n!)=1$,
 Theorem \ref{thm_main}.(2) shows that
 $\Sha_{div}(F, \PGL_n)=1$.
 This will be strenghthened in Corollary \ref{cor_pgl}.

 }
 \end{sremarks}

 \begin{scorollary} \label{cor_main_qs} 
 Let  $G$ be a semisimple quasi-split $F$--group
 which is quasi-split by a Galois prime to $p$ field extension. 
 Assume that $p$   does not divide the 
 order of $W_0$. In each case
 
\sm

(i) $G$ is split;

\sm

(ii) $G$ is simply connected
 or adjoint;

\sm 
 
 \noindent  we have $\Sha_{div}(F,G)=1$.
 \end{scorollary}

\begin{proof} Such an $F$--group is a rational $K$--variety
so the result follows from Theorem \ref{thm_main}.
\end{proof} 
 
 Corollary \ref{cor_main_qs2} of section \ref{sect_degree_one} will refine the statement  for groups without factor of type $E_8$.
 This applies as well to many other groups in view of the rationality  results of Chernousov-Platonov \cite{CP}
 and others. This is the case for example for an isotropic
 $F$-group of type $F_4$. 
 
 \medskip

 \section{Other Tate-Schafarevich sets} 
 \subsection{Various sets of valuations}
Given an affine algebraic $F$--group $G$, we can also 
study the  local-global principle for $G$-torsors
with respect to the following set of completions of $F$:

\sm

(1) Completions associated to the set $\Omega_F$ of 
of all non-trivial $B$-valuations on $F$ (i.e.,\ those whose valuation ring contains $T$).

\sm

(2) Completions associated to the set $\Omega^1_F$ of 
all non-trivial rank one $B$-valuations on $F$.

\sm

(3) Completions associated to the set $\Omega^1_{F,dvr}$ of 
all non-trivial rank one discrete  $B$-valuations on $F$.

\sm

We have $\Omega_{F,div} \subset 
\Omega^1_{F,dvr}  \subset \Omega^1_{F}  \subset \Omega_F$.
This leads to three  Tate-Shafarevich sets as for example
$\Sha_{\Omega^1_{dvr}}(F,G)=
\ker \Bigl( H^1(F,G) \to \prod_{v \in \Omega^1_{dvr}} 
H^1(F_v,G)\Bigr)$.
We have the following inclusions
$$
\Sha_{patch}(F,G) \subset \Sha_{\Omega_F}(F,G) \subset
 \Sha_{\Omega^1_F}(F,G)  \subset
\Sha_{\Omega^1_{dvr}}(F,G) \subset
\Sha_{\Omega^1_{div}}(F,G).
$$
\sm
According to \cite[Thm.\ 3.4]{HHKP}, we have 
$\Sha_{patch}(F,G) = \Sha_{\Omega^1_F}(F,G)=\Sha_{\Omega_F}(F,G)$.
The fact that $\Sha_{\Omega_F}(F,G) = \Sha_{\Omega^1_F}(F,G) = \Sha_{\Omega^1_{dvr}}(F,G)$ is known
in two following situations  \cite[Thm.\ 8.10, cases (i) and (iii)]{HHK2}:

\sm

(A) $G$ is a rational $F$--variety and the residue field $k$  of $T$ is 
algebraically closed of characteristic zero;

\sm

(B) $G$ is semisimple simply connected and the residue field $k$ of $T$ is 
algebraically closed of characteristic zero.
 
\sm

Coming back to the present paper,  Theorem \ref{thm_main} has then the following consequence.

\begin{scorollary} \label{cor_main}
Under the assumptions  of Theorem \ref{thm_main}, we have then
$$
\Sha_{patch}(F,G) =\Sha_{\Omega_F}(F,G) = \Sha_{\Omega^1_F}(F,G)  = 
\Sha_{\Omega^1_{dvr}}(F,G) = \Sha_{\Omega^1_{div}}(F,G)
$$
and those sets are trivial when $G$ is a retract  $F$-rational variety.
 \end{scorollary}

 \subsection{The constant case.}
 This is the case when we deal with a reductive $B$--group scheme $G$.
It  was  proven  in \cite[Thm.\ 4.2]{CTPS} 
(where only divisorial valuations are used in the proof)
that  $\Sha_{patch}(F,G) = \Sha_{\Omega^1_{dvr}}(F,G)$. 
Note that 
there is no need of any condition on the characteristic exponent $p$ of the residue field $k$. Other references are \cite[Thm.\ 8.10, case (ii)]{HHK2}
  and  \cite[Thm.\ 3.2]{CTOHHKPS2})
 This applies in particular when $G_F$ is a retract rational $F$--variety;
 since  $\Sha_{patch}(F,G)=1$, we have that $\Sha_{\Omega^1_{div}}(F,G)=1$ in
 characteristic free. The interest of our result is then 
 mostly for the non-constant case.

 \section{Variant with $0$-cycles of degree one}\label{sect_degree_one}
 
 The goal is  to replace rational points by $0$-cycles in order to
 improve Theorem \ref{thm_main}.(2).
 
 \subsection{Serre's injectivity question}\label{subsec_serre}
 For a semisimple (absolutely $k$--simple) group $G$ over a field $k$, Serre asked whether the 
 map
 $$
 H^1(k,G) \to \prod_{i=1}^c H^1(k_i,G)
 $$
 is injective where the $k_i$'s are finite field extensions of $k$
 such that the $g.c.d.$ of the $[k_i:k]$'s is prime to the torsion primes of $G$
 \cite[\S 2.4]{Se}. We can also ask the same question for 
 $G$ reductive assuming that the $g.c.d.$ of the $[k_i:k]$'s is $1$;
 note that the two questions are equivalent in the semisimple case
 in view of Lemma \ref{lem_torsion}.(2).
 If the above map is injective (resp.\ has trivial kernel),
 we will say that Serre's injectivity question (resp.\  Serre's 0-cycle question)
 has a positive answer.
Partial results are the following.

\smallskip

\begin{enumerate}
\item \label{serre1} The above questions reduce to the characteristic zero case
by using the lifting method \cite[proof of Thm.\ 3.4.1]{G4}.

 \item \label{serre2} The answer of Serre's 0-cycle question  is positive for
simply connected and adjoint classical groups 
and special orthogonal groups as established by Black \cite[Thm.\ 0.2, prop.\ 3.6]{Bl}
and for semisimple classical groups  involving types $A$, $B$, $C$ 
 (Bhaskhar, \cite{Bh}).

\item \label{serre3}  The answer of Serre's 0-cycle question  is positive
for quasi-split semisimple  groups without factors of type $E_8$.
This involves case by case results on other exceptional groups, see \cite[\S 5]{Bl},
\cite[Thm.\ 1.3]{Bh} and \cite[Cor.\ 3.4.5]{G4}. 
 
 \end{enumerate}

 \subsection{A local variant}

 \begin{sproposition} \label{prop_variant} 
 Let $G$ be a reductive
  $F$--algebraic group. Let $l$ be a prime, $l \not =p$.

\smallskip

(1) Let $\gamma \in \Sha_{\Omega^1_{div}}(F,G)$. Then there exists 
a finite extension $L/F$ of degree prime to $l$
such that   $\gamma_L \in \Sha_{patch}(L,G)$.

\smallskip

(2) Furthermore if $\Sha_{patch}(L,G)=1$, then 
$\gamma_L=1$.
\end{sproposition}

 \begin{proof} 
 (1) We are given a class $\gamma =[E] \in \Sha_{\Omega^1_{div}}(F,G)$. 
 The proof is quite similar to that of Theorem \ref{thm_main}
 by using the localized version of the technique of
 reduction to torsors to finite subgroup schemes (established 
by Theorem \ref{thm_cgr2}).
 
 We remind the reader that $G_0$ denotes the Chevalley form of $G$.
We start with the isomorphism class $[G] \in H^1(F,\Aut(G_0))$
by using the sequence \cite[XXIV.1.3]{SGA3}

\begin{equation}\label{eq_aut}
1 \to G_{0,ad} \to \Aut(G_0) \to \Out(G_0) \to 1
\end{equation}
where $\Out(G_0)$ is a constant $\ZZ$--group scheme.
Since $H^1(F,\Out(G_0))= \break \Hom_{ct}\bigl( \Gal(F_s/F), \Out(G_0)(\ZZ) \bigr)/\sim$, the image of $[G]$ is represented by 
a continuous homomorphism $u : \Gal(F_s/F) \to  \Out(G_0)(\ZZ)$
whose kernel defines a Galois extension $F'/F$ 
which is the minimal subextension of $F_s$
making $G$ an inner form. Our assumption implies
that $\Gal(F'/F)$ is of prime to $p$ degree
and so  that $\Gamma=u(\Gal(F_s/F) )$ is a finite subgroup
of $\Out(G_0)(\ZZ)$. 
We denote by  $\Aut_{\Gamma}(G_0)=  \Aut(G_0) \times_{\Out(G_0)} \Gamma$; by construction $[G]$
arises from a cohomology class $\beta \in H^1(F,\Aut_\Gamma(G_0))$.
According to Theorem \ref{thm_cgr2} applied to  $\ZZ$-group scheme $\widetilde G_0=\Aut_\Gamma(G_0)$
there exists a finite $\ZZ$--subgroup $S^\sharp$ of
$\Aut_\Gamma(G_0)$
 which is  finite free of index $l^c$ and such that there exists a finite  prime to $l$
field extension $F'/F$ such that $\beta_{F'}$ belongs to the
 image of $H^1(F',S^\sharp) \to H^1(F',\Aut_\Gamma(G_0))$.
 Once again, replacing $F$ by $F'$ permits to suppose 
 that  $\beta$ belongs to the
 image if $H^1(F,S^\sharp) \to H^1(F,\Aut_\Gamma(G_0))$.

 The same method works as well for the 
 class $[(E,G)] \in H^1( F, G_0 \rtimes \Aut(G_0))$
 by using the exact sequence
 $1 \to G_0 \rtimes_\ZZ G_{0,ad}  \to G_0 \rtimes_\ZZ \Aut(G_0)
 \to \Out(G_0) \to 1$ and provides a finite 
$B$--subgroup of $S^{\diamond}$ of  $(G_0 \rtimes_\ZZ \Aut(G_0))_B$
of $l$-power  such that the $G_0 \rtimes_\ZZ \Aut(G_0)$-torsor
over $\Spec(F)$  arises from an $S^\diamond$--torsor (again up to taking a further 
finite prime to $l$--extension of $F$).

 Proposition \ref{prop_main}, $(ii) \Longrightarrow (i)$, applies and shows that  there  exists a regular proper model $\gX$ of $X$ 
with special fiber $Y=\gX_k$ such that for every point
$y \in Y$, then $E(F_y) \not = \emptyset$.
In other words $[E]$ belongs in $\Sha_{\gX}(F,G)$
and a fortiori to $\Sha_{patch}(F,G)$.

\sm

\noindent (2)  This readily follows of (1).
 \end{proof}

\newpage
 
 \subsection{The $0$-cycle variant}

 \begin{stheorem} \label{thm_main_cycle} 
 Let $G$ be a reductive
  $F$--algebraic group
 and let $G_0$ be the underlying Chevalley  reductive $\ZZ$--group scheme.
 We assume that $\Sha_{patch}(L,G)=1$
 for all finite extensions $L$ of $F$.
Let $E$ be a $G$-torsor over $\Spec(F)$ such that 
$[E] \in \Sha_{\Omega^1_{div}}(F,G)$.

\smallskip

(1)  $E$ admits a $0$--cycle of degree $p^c$.

\smallskip

(2) Assume furthermore that $p$ is not  a torsion prime of  $DG_0$ and that $G$ becomes an inner form of $G_0$
after a  Galois prime to p extension of $F$. Then
$E$ admits a $0$--cycle of degree $1$.

\smallskip

(3) Assume
furthermore that Serre's 0-cycle question admits a positive
answer for $G$. Then the $G$--torsor $E$ is trivial.

\end{stheorem}

 \begin{proof}
(1) Let $[E]  \in \Sha_{\Omega^1_{div}}(F,G)$. For each prime $l \not = p$,
 we have to prove that there exists a prime to $l$ finite field extension 
 $F'/F$ such that  $E(F') \not = \emptyset$.
This fact follows indeed of Proposition \ref{prop_variant}.(2).

  \sm

\noindent (2) From (1),
 it remains to deal with the case $l=p$.
 Up to replace $F$ by  a suitable  prime to $p$-extension of $F$,
we can assume that $G$ is an inner form of $G_0$.
In other words $G$ is the twist of $G_0$ by a $G_{0,ad}$--torsor
$Q$. Let $\Ind_F(Q)$ be the index of $Q$, that is, the 
$g.c.d.$ of the degrees of the rational points of the $F$-variety $Q$.
Since $p$ is not a torsion prime of  $G_{0,ad}$, 
Lemma \ref{lem_torsion}.(2) shows that $p$ does not divide
$\Ind_F(Q)$ so that $Q$ carries a point after a prime to $p$ extension.
Once again it is harmless to assume that $G=G_0$.

Next we use the exact sequence $1 \to C_0 \to G_0 \to G_0/C_0 \to 1$
where $C_0$ is the radical torus of $G_0$.
Then $C_0$ is split and Hilbert 90 theorem  provides a commutative exact 
diagram of pointed sets 
\[\xymatrix{
 1  \ar[r]  & H^1(F,G_0)   \ar[r]  \ar[d] & H^1(F,G_0/C_0) \ar[d] \\
 1  \ar[r]  & \prod_{v \in \Omega_{F,div} } 
 H^1(F_v,G_0)   \ar[r]  & \prod_{v \in \Omega_{F,div} } H^1(F_v,G_0/C_0).
}\]

\noindent The above diagram
 reduces the proof  $E$ admits a 0-cycle of degree $1$ to the 
 case of an $F$-torsor $Q'$ under the semisimple split group $G'_0=G_0/C_0$.
The  argument used for  $Q$ and $G_{0,ad}$ works and enables
us to conclude that  $E$ becomes trivial  after a prime to $p$ extension of $F$.
\smallskip

\noindent (3) This is straightforward.
 \end{proof}

 \begin{sremarks}{\rm 
 (a)  By inspection of the proof, 
 it is enough to assume that there exist finite fields extensions $F_i$'s of $F$
 such that $g.c.d.( [F_i:F])=1$ and such that $\Sha_{patch}(L_i,G)=1$
 for each finite field extension $L_i/F_i$.

 \sm
 
\noindent  (b) 
 The assumption on the triviality of the sets $\Sha_{patch}$
 is satisfied when $G$ is a retract rational $F$--variety
 \cite[Thm.\ 2.2.4]{K}.

 \sm
 
\noindent  (c) There are cases when  $G$ is not  retract rational
but for which the sets $\Sha_{patch}$ vanish. We provide an example below, 
see Remark \ref{rem_PPS}.
 
  }
 \end{sremarks}

 \subsection{Applications}
 
 Taking into account the results of \S \ref{subsec_serre},
 Theorem \ref{thm_main_cycle}.(3) has the following consequence
 which slightly improves Corollary \ref{cor_main_qs}.

 \begin{scorollary} \label{cor_main_qs2} 
 Let  $G$ be a semisimple quasi-split $F$--group with no $E_8$ factor.
 Assume that $p$ is not a torsion prime of $G$. In each case
 
\sm

(i) $G$ is split;

\sm

(ii) $G$ is simply connected
 or adjoint;

\sm 
 
 \noindent we have $\Sha_{\Omega^1_{div}}(F,G)=1$.
 \end{scorollary}

 \begin{proof}  In both cases, we have seen at the end of \S \ref{subsec_serre}
 that  Serre's 0-cycle question admits a positive
answer for $G$. Then Theorem \ref{thm_main_cycle}.(3) implies that 
 $\Sha_{\Omega^1_{div}}(F,G)=1$.
 \end{proof}

We deal now with special cases where $G$ is a retract rational 
$F$-variety and when Serre's 0-cycle question has a positive answer.
We recover then by our method results obtained by 
using injectivity properties of cohomological invariants
and local-global principles for Galois cohomology.
Parimala-Suresh's theorem \cite[Cor. \. 4.4]{HKP} states
that the map 
\begin{equation}\label{eq_PS}
H^{q+1}(F, \mu_n^{\otimes q}) \to  \prod_{v \in \Omega^1_{div}} H^{q+1}(F_v, \mu_n^{\otimes q})
\end{equation}
is injective for each $q \geq 1$ and $(n,p)=1$.
For $\PGL_n$, since $H^1(F,\PGL_n)$ injects in $H^{2}(F, \mu_n)$
we obtain

\begin{scorollary} \label{cor_pgl} Let $n$ be an integer prime to $p$.
Then $\Sha_{\Omega^1_{div}}(F, \PGL_n)=1$.
\end{scorollary}

It is also a corollary of Corollary \ref{cor_main_qs2}.(1).
Similarly we have the following statement in view
of the Merkurjev-Suslin's theorem on the characterization 
of reduced norms of squarefree central simple algebra
\cite[Thm.\ 8.9.1]{GSz}.

\begin{scorollary} \label{cor_nrd} (Inner type $A$) Let $A$ be a 
central simple $k$--algebra with squarefree index prime to $p$.
Then  $\Sha_{\Omega^1_{div}}(F,\SL_1(A))=1$.
\end{scorollary}

Let us explain our way to prove it.

\begin{proof} We write $A=\mathrm{M}_r(D)$ for a division $F$--algebra $D$.
We have $H^1(F,\SL_1(A))=F^\times/ \Nrd(A^\times)$
and 
as noticed in \eqref{serre2} of \S \ref{subsec_serre},  the answer to Serre's injectivity
question is positive for  $\SL_1(A)$-torsors.
According to \cite[prop.\ 2.4]{G2}, $\SL_1(A)=\SL_r(D)$ is retract $F$--rational.
Theorem \ref{thm_main_cycle}.(3) applies and yields $\Sha_{\Omega^1_{div}}(F,\SL_1(A))=1$.
\end{proof}

\begin{sremark} \label{rem_PPS} {\rm 
Assume that $K$ is a $p$-adic field and that $A$ is of index $n$ prime to $p$.
Parimala-Preeti-Suresh have proven that  $\Sha_{\Omega^1_{dvr}}(F,\SL_1(A))=1$
\cite[Thm.\ 1.1]{PPS}. If the index divides $4$, we know that the 
$F$--group $\SL_1(A)$ is not retract $F$-rational \cite{Me}.
}
\end{sremark}

\begin{scorollary}  \label{cor_D} Let $A$ be a central simple $F$--algebra of 
degree $2n \geq 4$ equipped with an orthogonal 
involution $\sigma$. We assume that $p \not =2$ and consider
the special orthogonal $F$--group $O^+(A, \sigma)$ \cite[\S 23]{KMRT}. 
Then we have  $\Sha_{\Omega^1_{div}}\bigl(F,O^+(A, \sigma)\bigr)=1$.
\end{scorollary} 
 
 In the special case of $\SO(q)$
this  was known by applying  \cite[thm 4.3.(i)]{PS} to $\SO_{2n}$ and by inspecting the proof to check that only divisorial valuations are used; if $q$ is isotropic
we have furthermore that $\Sha_{\Omega^1_{div}}(F,\Spin(q))=1$
by using the fact that the map $H^1(F,\Spin(q)) \to H^1(F, \SO(q))$ has trivial kernel.
We present to the  proof of Corollary \ref{cor_D}.
 
 \begin{proof}  It uses the answer to Serre's injectivity
question is positive for  torsors; more precisely
 this is \cite[prop.\ 3.6]{Bl}.
Next  the $F$-group $O^+(A, \sigma)$ is $F$-rational 
 in view of the Cayley parametrization \cite[p.\ 599]{W} so that 
 Theorem \ref{thm_main_cycle}.(3) applies.
 \end{proof}

 In the case of tori, Serre's injectivity question has 
 a positive answer in view of the classical corestriction-restriction argument.

 \begin{scorollary} \label{cor_main_tori} 
 Let  $T$ be an $F$-torus which is split by a prime to $p$ Galois extension. 
 Assume that $T$ is a retract rational $F$--variety.
 We have $\Sha_{\Omega^1_{div}}(F,T)=1$.
 \end{scorollary}

 \appendix
 
\section{Torsion primes} 
Let $k$ be a field. If $X$ is a non-empty $k$--scheme locally of finite type, we denote  by
$\Ind_k(X)$ its index, i.e.\ the g.c.d.\ of 
the degrees of the finite  field extensions $k'$ of $k$
such that $X(k') \not= \emptyset$.
 In other words $\Ind_k(X)$ is the smallest positive  
  degree of a $0$-cycle on $X$.
If $k'/k$ is a finite field extension, it satisfies the 
divisibility properties  $\Ind_k(X) \, \mid \, \Ind_{k'}(X_{k'})
\, \mid \, [k':k] \, \Ind_k(X)$.

 Let $M$ be a $k$--group of multiplicative type
 and let $\alpha \in H^i_{\fppf}(k,M)$ with $i \geq 1$.
The group  $H^i_{\fppf}(k,M)$ is torsion
and the period of $\alpha$ is its exponent;  it is denoted
 by $\per_k(\alpha)$. Its index $\Ind_k(\alpha)$ is
 the g.c.d.\ of 
the degrees of the finite  field extensions $k'$ of $k$
such that $\alpha_{k'}=0$. It satisfies similarly the above 
divisibility properties, i.e.\ $\Ind_k(\alpha) \, \mid \, \Ind_{k'}(\alpha_{k'}) \, \mid \, [k':k] \, \Ind_k(\alpha)$
by using the corestriction maps as defined in \cite[\S 0.4]{G1}.

 In the case  of $\Br(k)=H^2(k,\GG_m)$,
this is related with the index of central simple algebras, see \cite[Prop.\ 4.5.1]{GSz}.

\begin{slemma}\label{lem_period}

(1) Assume that $\per_k(\alpha)=l^r$ for a prime $l$.
Then  $\Ind_k(\alpha)$ is a power of $l$. 

\sm

(2) The period $\per_k(\alpha)$ divides
 $\Ind_k(\alpha)$ and these numbers share the same prime divisors.
\end{slemma}

\begin{proof} 
(1) We assume first that $l$ is invertible in $k$.

\sm

\noindent{\it Case $i=1$ and $M$ is finite.}
It is harmless to assume that $M$ is $l$--primary and
in particular \'etale.
We see $\alpha \in H^1_{\fppf}(k,M)$ as the class of 
an $M$--torsor $X$.
Since $M$ is \'etale of degree $l^a$
 so is $X$. Then the identity point  of $X(k[X])$ provides  an effective (separable) $0$-cycle 
of degree $l^a$ on $X$. 

\sm

\noindent{\it Case $i=1$ and  $M$ is a torus.}
By using  the exact sequence $1 \to {_{l^r}M} \to M 
\xrightarrow{\times l^r} M \to 1$, we
see that $\alpha$ comes from a class $\beta \in H^1_{\fppf}(k,{_{l^r}M})$.
The preceding case shows  that $\Ind_k(\beta)$ is  a power of 
$l$ and so is $\Ind_k(\alpha)$.

\sm

\noindent{\it Case $i=1$.} We have an exact sequence $1 \to T \to M \xrightarrow{f} N \to 1$ where $N$ is finite and $T$ is a $k$-torus.
From the first case, we know that $\Ind_k(f_*\alpha)$ is an $l$-power, 
and more precisely that  there are finite separable field
extensions  $k_i$'s of $k$  killing $f_*\alpha$ such that $g.c.d.( [k_i:k])$ is 
an $l$--power. For each $i$ the preceding case provides
separable field extensions  $l_{i,j}$'s of $k_i$
killing $\alpha_{k_i}$ (which arises from a 
class of $H^1(k_i,T)\{l\}$) and such that 
$g.c.d.( [l_{i,j}:k_i])$ is  an $l$--power $l^{e_i}$. 
Without loss of generality we can assume that 
we deal with  $l_{i,j}$'s for $i=1, \dots, n$ and $j=1,\dots, m$. 

The $l_{i,j}/l$'s kill $\alpha$
and we claim that  $g.c.d.([l_{i,j}:k])$ is an $l$--power.
Let $q \not = l$ be a prime dividing $[l_{i,j}:k] = [l_{i,j}:k_i] [k_i:k]$
for all $i,j$. Then $q$ divides $l^{e_i} [k_i:k]$
so divides each $[k_i:k]$ which is a contradiction.
Thus $\Ind_k(\alpha)$ is a power of $l$.

\sm

\noindent{\it  $i\geq 2$.} It goes by the classical shifting argument.
Let $K/k$ be a finite separable such that $\alpha_K=0$
and  consider the exact sequence $1 \to M \to R_{K/k}(M) \to Q \to 1$
of $k$--groups of multiplicative type. Since $H^i_{\fppf}(k,R_{K/k}(M))=H^i_{\fppf}(K,M)$,
$\alpha$ comes from a class $\gamma \in H^{i-1}_{\fppf}(k,Q)\{l\}$ with respect to 
the boundary map $\partial : H^{i-1}_{\fppf}(k,Q) \to H^i_{\fppf}(k,M)$.
By induction on $i$, we obtain  that $\Ind_k(\gamma)$ is a power of 
$l$ and so is $\Ind_k(\alpha)$.

\sm
We come now to the case when $l=p$ is the characteristic of $k$.

\sm

\noindent{\it Case $M$ is finite.}
We can assume that $M$ is $p$--primary so that 
$M(\ol{k})=1$. Since $H^i(\Gal(\ol{k}/k^{perf}), M(\ol{k}))
\simlgr H^i_{\fppf}(k^{perf},M)$, it follows that 
$\alpha_{k^{perf}}=1$. There exists then 
a finite purely inseparable field extension of $k$
which kills $\alpha$. Thus $\Ind_k(\alpha)$ divides a power of $p$.

\sm

\noindent{\it General case.}
The above arguments (by exact sequences and shifting) yield
the torus case and the general case.

\sm

\noindent (2) For each finite field extension $k'/k$,
we have a corestriction  map $\Cor_k^{k'}: H^i(k',M) \to H^i(k,M)$
which satisfies $\Cor_k^{k'} \circ \Res_k^{k'}= \times [k':k]$,
see \cite[\S 0.4]{G1}.
If $k'$ splits $\alpha$, it follows that $[k':k] \alpha=0$
so that  $[k':k]$ divides $\per_k(\alpha)$.
It follows that $\per_k(\alpha)$ divides
 $\Ind_k(\alpha)$.
It remains to show that a prime divisor $l$ of 
$\Ind_k(\alpha)$ divides $\per_k(\alpha)$.
We consider the prime decompositon  $\per_k(\alpha)=p_1^{r_1} \dots p_c^{r_c}$  and write $\alpha= \alpha_1 + \dots+ \alpha_c$ 
with $\per_k(\alpha_i)=p^{r_i}$.
Given a finite field extension $k'/k$,
$k'/k$ kills $\alpha$ iff it kills each $\alpha_i$.
It follows that $\Ind_k(\alpha)$ is the g.c.d.\ of the 
$\Ind_k(\alpha_i)$'s.
We are then reduced to the case  $\per_k(\alpha)=l^{r}$ for a prime $l$
which has been handled in (1).

\end{proof}

\begin{sremark}{\rm
By inspection of the proof, in (1), if $l$ is invertible
in $k$, the index is realized by separable field extensions.
A contrario, if $l=p$ is the characteristic of $k$,
the index is realized with a single purely inseparable field extension.
}
\end{sremark}

Given a semisimple  algebraic group $G$ defined over 
 $k$, we remind the reader that we can attach
 to $G$ a list of torsion primes \cite[\S 4.8, 5.1]{G4}
 which depends only of the isogeny class of $G$
and which extend  Serre's definition \cite[\S 2.4]{Se}.   
The index $\Ind_k(G)$ of $G$ is the g.c.d.\ of
the degrees of the finite splitting fields extensions of $G$, that is, making $G$ split. In other words, if $G_0$ is Chevalley form
of $G$, this is the index of  the $\Aut(G_0)$-torsor
$\Isom(G_0,G)$.

\begin{slemma} \label{lem_torsion}
(1) If $G$ is almost $k$-simple,  the primes dividing $\Ind_k(G)$ are torsion primes.

\sm
\sm

\noindent (2) If $E$ is a $G$--torsor, the 
primes dividing $\Ind_k(E)$ are torsion primes.

\sm

\noindent (3) Let $X$ be the variety of Borel $k$--subgroups of $G$.
Then the primes dividing $\Ind_k(X)$ are torsion primes.

\end{slemma}

\begin{proof}
(1) This follows of  \cite[prop.\ A1]{T2}.

\smallskip

\smallskip

\noindent (2)  We assume first that $G$ is 
semisimple adjoint and quasi-split. Given $\gamma \in H^1(k,G)$, 
Steinberg's theorem provides a maximal $k$--torus $T$
such that $\gamma$ belongs to the image of 
$H^1(k,T) \to H^1(k,G)$ \cite[Thm.\ 11.1]{St}. Let $d$ be the period of 
an antecedent of $\gamma$.
Then $\gamma$ belongs to the image $H^1(k,{_dT}) \to H^1(k,G)$
so that $\Ind_k(\gamma)$ divides a power of $d$ 
according to Lemma \ref{lem_period}.
 In view of a result of Harder
\cite[Thm.\ 5.1.2]{G4}, the primes dividing $d$ are torsion primes.
Thus the primes dividing $\Ind_k(\gamma)$ are torsion primes.

The next case is when $G$ is adjoint.
We denote by $G^{qs}$ the quasi-split form of $G$.
Since $G$ is  inner form of  $G^{qs}$, the first 
case shows that there exists finite field extensions
$k_i$'s of $k$ which quasi-split $G$ and such that 
the prime divisors of $g.c.d.( [k_i:k])$ are torsion primes.
Let $E$ be a $G$--torsor. For each $i$, the first case provides
 finite field extensions
$l_{i,j}$'s of $k_i$ which split $E_{k_i}$ and such that 
the prime divisors of $g.c.d.( [l_{i,j}:k_i])$ are torsion primes.
It follows that the $l_{i,j}$'s split $E$ and
that $g.c.d.( [l_{i,j}:k])$ is a product of  torsion primes
(the argument is similar with that in the proof of Lemma \ref{lem_torsion}.(2)).

For the general case, we consider the exact sequence
$1 \to \mu \to G \to G_{ad} \to 1$.
Let $E$ be a $G$--torsor. The preceding  case provides
 finite field extensions
$k$'s of $k$ which split $E_{ad}= E/\mu$ and such that 
the prime divisors of $g.c.d.( [k_i:k])$ are torsion primes.
For each $i$, there exists a $\mu$--torsor $F_i$
such that $E_{k_i} \cong F_i \wedge^{\mu} G$.
Since the order of $\mu$ is a product of torsion primes, 
for each $i$, there are finite field extensions
$l_{i,j}$'s of $k_i$ which split $F_{k_i}$ (and then 
a fortori $E_{k_i}$) and such that 
the prime divisors of $g.c.d.( [l_{i,j}:k_i])$ are torsion primes
(Lemma \ref{lem_period}).
It follows that the $l_{i,j}$'s split $E$ and
that $g.c.d.( [l_{i,j}:k])$ is a product of  torsion primes.

\sm

\noindent (3) It is proved  along the same lines as (2). 

\end{proof}

 \section{0-cycles for twisted flag varieties}
 We can also discuss the following variant of the main result of \cite[Thm. 5.4]{GP}.
 The setting is that of \S \ref{subsec_setting}. 
The ring $B$ is a  complete  discrete valuation ring with fraction field $K$, residue field $k$ and uniformizing parameter $t$; the field  $F$ is a one-variable function field over $K$.

 \begin{stheorem} \label{thm_patching} Let $G$ be a semisimple
 $F$--group. Let $Z$ be a twisted flag $F$--variety of $G$
 such that $Z(F_v) \not = \emptyset$
 for all discrete valuations of $F$ arising from
 normal models of $X$.

  \sm

 \noindent (1) $Z$ admits a $0$--cycle of degree $p^c$.

  \sm

 \noindent (2) Assume that $p$ is not a torsion prime of $G$
 and that $G$ becomes quasi-split after a prime to $p$ Galois extension of $F$.
  Then the $F$-variety $Z$ admits a $0$-cycle of degree 1.
 \end{stheorem}
 
 \begin{proof}
 Without loss of generality we can assume that $G$ is adjoint.
   Let $G_0$ be the Chevalley form of $G$ and
   let $(B_0,T_0)$ be a
   Killing couple for  $G_0$ and let $\Delta_0$ be
   the associated Dynkin diagram.
 The variety $Z$ is a form of the variety $\Par_I(G_0)$
  of parabolic subgroups of type $I$ where $I \subset \Delta_0$ is stable
   under the star action defined by the $\Aut(G_0)$--torsor
   $Q=\mathrm{Isom}(G_0,G)$.
   In particular $Q$ admits a reduction $Q_I$ to the stabilizer
   $\Aut_I(G_0)$ for the action  $\Aut(G_0)$ on $\Delta_0$ through
the morphism $\Aut(G_0) \to \Out(G_0)\simlgr \Aut(\Delta_0)$.
Furthermore $Z$ is isomorphic to  $^{Q_I} \! \Par_I(G_0)$.

 \smallskip

 \noindent (1) For each prime $l \not =p$, we have to
 show that there exists a finite prime to $p$ field extension
 $L/F$ such that $Z(L) \not= \emptyset$.
We apply  now Theorem \ref{thm_cgr2}.(2)
to the $\ZZ$--group scheme $\Aut_I(G_0)$.
It provides a finite $\ZZ$--subgroup $S^\sharp$
of degree $l^m$ of $\Aut_I(G_0)$ and a finite extension $L$ of
$F$ of degree prime to $l$ such that $Q_{I,L}$
admits a reduction to  $S^\sharp$, that is, there exists
a $S^\sharp$-torsor $E$ over $L$ such that $E \wedge^{S^\sharp} \Aut_I(G_0) \cong Q_I$.
It is harmless to replace $L$ by $F$.
We have that  $S^\sharp_B$ is finite \'etale of degree $l^m$.

The  finite flat $\ZZ$--group scheme $S^\sharp$ admits
a faithful representation $S^\sharp \hookrightarrow \GL_{N,\ZZ}$
   \cite[\S 1.4.5]{BT} and the quotient
   $\GL_{N,\ZZ} /S^\sharp$ is representable by an affine $\ZZ$-scheme \cite[\S III.2.6]{DG}.
  According to \cite[Lemma 5.1]{GP},
  there exists  a regular model $\gX$ of $X$ and
  a strict normal crossing divisor  $D$ containing  the irreducible components of $Y$ such that $E$ extends
  to a $\gX \setminus D$--torsor $\gE$  under $S^\sharp$.
  We put $\gQ_I= \gE \wedge^{S^\sharp}\Aut_I(G_0)$ and consider
  the $\gX \setminus D$-group scheme $\gG= {^{\gQ_I}\!G_0}$ with generic fiber $G$.
  
 If $y\in Y$ is of codimension one, by hypothesis, $Z(F_y)$ is not empty. 
 We therefore look at a closed point $P \in Y$ and pick a height one prime $\cp$ in $\widehat R_P$ that contains
  $t$. It  defines a branch of $Y$ at $P$ lying
 on some irreducible component $Y_1$ of $Y$.

We consider the local ring $A=R_P$ of $\gX$ at $P$ and denote by $A_D$
its localization at $D$. Since $S^\sharp_B$ is finite \'etale of degree prime
to $p$, $H^1(A_D,S^\sharp)$ consists in loop torsors as defined in \cite[\S 2.3, lemma 2.3.(2)]{G5},  i.e. those arising from cocycles related to tame Galois covers of $A_D$. It follows that the $A_D$--torsor
$Q_I$ is a loop $\Aut_I(G)$--torsor \cite[lemma 2.3.(3)]{G5},
so that $\gG \times_{\gX \setminus D} A_D$ is
by definition a loop reductive group scheme.

Let $F_{P,v}$ be the completion of the field $F_P$ for the valuation associated to the blow-up of $\Spec(A)$ at its closed point.
Our assumption states in particular that $Z(F_{P,v}) \not = \emptyset$,
that is, $G_{F_{P,v}}$ admits a parabolic subgroup of type $I$.
According to \cite[th.\ 4.1, (iii) $\Longrightarrow$ (i)]{G5}, $\gG \times_{\gX \setminus D} A_D$
admits a parabolic subgroup of type $I$.
A fortiori $G_{F_P}$ admits a parabolic subgroup of type
$I$ so that $Z(F_P) \not = \emptyset$.
In view of \cite[Cor.\ 4.5]{GP}, we conclude that $Z(F)$ is not empty.

 \smallskip

 \noindent (2) After (1), it remains to deal with the case $l=p$.
 In the proof of Theorem \ref{thm_main_cycle}, we have seen
 that our assumption on $p$ implies that $G$ becomes split after
  a prime to $p$ extension $L/F$. It follows that $Z(L) \not = \emptyset$.
 \end{proof}

In several cases, we know that is the $F$-variety $Z$ admits a $0$-cycle of degree 1,
then it admits  $F$-point. This provides the following improvement of \cite{GP}
which is close of Corollary  \ref{cor_main_qs2}.

\begin{scorollary}
Let $G$ be a semisimple $F$-group and assume that
$p$ is not torsion prime of $G$ and that $G$ becomes
quasi-split after a prime to $p$ Galois extension.
 Let $Z$ be the variety of Borel subgroups of $G$. Assume that
$G$ has no factor of type $E_8$. If $Z(F_v) \not = \emptyset$
 for all discrete valuations of $F$ arising from
 normal models of $X$, then $Z(F) \not = \emptyset$.
 \end{scorollary}

An important case is the following.

\begin{scorollary} \label{cor_hyp}
Assume that $p \not =2$.

\sm
\noindent (1) Let $q$ be a regular quadratic form of  rank $\geq 3$.
If $q_{F_v}$  is isotropic  
 for all discrete valuations of $F$ arising from
 normal models of $X$, then $q$ is isotropic.

\sm

\noindent (2) 
Let $A$ be a central simple $F$--algebra
of degree $2n$ equipped with an orthogonal
involution $\sigma$.
If $\sigma_{F_v}$  is hyperbolic  
 for all discrete valuations of $F$ arising from
 normal models of $X$, then $\sigma$ is hyperbolic.
 \end{scorollary}

\begin{proof}
(1) In this case, we apply the result
to the projective quadric $Z$ which is projective
homogeneous under $\SO(q)$.
Theorem \ref{thm_patching}.(1) shows that $Z$
admits a $0$-cycle of degree $p^c$ which is odd.
It follows that $q$ become isotropic on 
some finite field extension $F'/F$ of odd degree.
According to Springer's theorem $q$ is 
isotropic. 

\sm

\noindent (2)
If $A$ is split, $\sigma$ is adjoint to 
a  regular quadratic $F$-form $q$ of dimension
$2n$. Witt's decomposition and (1) yield
that local hyperbolicity implies hyperbolicity 
for the quadratic form $q$.

If $A\cong \mathrm{M}_r(D)$ is not split, we consider the $F$--variety of isotropic ideals of $(A,\sigma)$ of  dimension $2n$. This is a projective $G$--variety
for the $F$-group $G=\mathrm{SO}(A, \sigma)$.
Theorem \ref{thm_patching}.(1) shows that $Z$
admits a $0$-cycle of degree $p^c$ which is odd.
It follows that $\sigma$ become hyperbolic on 
some finite field extension $F'/F$ of odd degree.
Bayer-Lenstra's theorem (see \cite[Cor.\ 6.15]{KMRT})
enables us to conclude that 
$\sigma$ is hyperbolic.
\end{proof}

In general flag varieties may have a 
$0$-cycle of degree one but no rational point \cite{P}. However there are a few
cases where we know that the existence of a
$0$-cycle of degree is equivalent to 
the existence of a rational point.
The paper \cite{GS} presents the case of 
some flag variety in type $E_7$ (resp.\ $E_6$)
and we have then a local-global principle
as in Corollary \ref{cor_hyp} when $p \not =2$ (resp.\
$p \not =3$).

\sm
Finally it would be interesting to  investigate  local-global principles 
for 0-cycles on projective $G$--homogenous varieties in the spirit of the paper  \cite{CTOHHKPS1}
which deals mostly with the patching viewpoint.
For us, after localization at a prime $l$, 
we are interested in local-global principles 
for 0-cycles for closed points (resp.\ closed separable points)
 of degree prime to $l$ with respect with completions for divisorial valuations.

\bigskip

\medskip

\end{document}